\documentclass[11pt]{article}
\usepackage{graphicx}
\usepackage{amsfonts}
\usepackage{amsmath}
\usepackage{amsthm}
\usepackage{amssymb}
\usepackage{fancyhdr}
\usepackage{indentfirst}
\usepackage{titlesec}
\usepackage{hyperref}
\usepackage{abstract}
\newcommand{\al}{\alpha}

\newcommand{\R}{\mathbb{R}}
\newcommand{\N}{\mathbb{N}}

\newcommand{\lan}{\langle}
\newcommand{\ran}{\rangle}

\newcommand{\La}{\Lambda}
\newcommand{\si}{\sigma}

\newcommand{\de}{\delta}

\newcommand{\ep}{\epsilon}
\newcommand{\pr}{\prime}
\newcommand{\Om}{\Omega}
\newcommand{\Ga}{\Gamma}
\newcommand{\ga}{\gamma}

\newcommand{\ti}{\tilde}

\newcommand{\Length}{\textrm{Length}}

\renewcommand{\baselinestretch}{1.1}
\begin{document}

\newtheorem{theorem}{Theorem}[section]
\newtheorem{proposition}[theorem]{Proposition}
\newtheorem{corollary}[theorem]{Corollary}

\newtheorem{claim}{Claim}

\theoremstyle{remark}
\newtheorem{remark}[theorem]{Remark}

\theoremstyle{definition}
\newtheorem{definition}[theorem]{Definition}

\theoremstyle{plain}
\newtheorem{lemma}[theorem]{Lemma}

\numberwithin{equation}{section}

%
%\renewcommand{\abstractname}{\large \CJKfamily{hei}摘\;\;要}
%
%\titleformat{\section}[hang]
%   {\normalfont\bfseries\Large\CJKfamily{hei}}
%   {\CJKnumber{\arabic{section}}}
%   {12pt}{\Large}
%
%\titleformat{\subsection}[hang]
%  {\normalfont\bfseries}%\large}
%   {\arabic{section}.\arabic{subsection}}
 %  {11pt}{\large}
 \titleformat{\section}%[runin]
   {\normalfont\bfseries\large}%\filcenter}
   {\arabic{section}}
   {12pt}{}
 \titleformat{\subsection}[runin]
   {\normalfont\bfseries}
   {\arabic{section}.\arabic{subsection}}
   {11pt}{}
%
%\renewcommand{\figurename}{图}
%
%\renewcommand{\tablename}{表}
%
%\renewcommand{\refname}{参考文献}

%Page head and Page foot
\pagestyle{headings}
%\lhead{} \chead{min-max minimal surfaces} \rhead{}
%\lfoot{} \cfoot{\thepage} \rfoot{}
\renewcommand{\headrulewidth}{0.4pt}

%Tile and Author
\title{\textbf{On the free boundary min-max geodesics}}
\author{Xin Zhou\footnote{The author is partially supported by NSF grant DMS-1406337.}\\}
%\date{\today}
\maketitle

\pdfbookmark[0]{}{beg}

%Begin abstract
\renewcommand{\abstractname}{}    % clear the title
\renewcommand{\absnamepos}{empty} % originally center
\begin{abstract}
\textbf{Abstract:} Given a Riemannian manifold and a closed submanifold, we find a geodesic segment with free boundary on the given submanifold. This is a corollary of the min-max theory which we develop in this article for the free boundary variational problem. In particular, we develop a modified Birkhoff curve shortening process to achieve a strong ``Colding-Minicozzi" type min-max approximation result.
\end{abstract}

%Begin main text

\section{Introduction}\label{introduction}

Let $(M^{n}, g)$ be a complete and homogeneously regular Riemannian manifold, and $N^{m}$ a closed submanifold. We consider the problem of finding a geodesic in $M$ with end points on $N$, which meets $N$ orthogonally. This is a variational problem with free boundary conditions. We are interested in the case $\pi_{1}(M, N)=0$, where every curve with end points on $N$ can be shrunk to a point. Hence direct variational method fails in this situation, and we explore the min-max methods. Similar idea was first given by Birkhoff in the 1910s to find closed geodesics on the 2-sphere \cite[V.7]{B} (See \cite[\S 5]{CM11}\cite[\S 2]{Cr88} for more details). In brief, given a sweep-out, i.e. 1-parameter family of closed curves which cover the 2-sphere, Birkhoff developed a curve shortening process to make each slice of the sweep-out as tight as possible, and then obtained a closed geodesic as the limit of slices with maximal length. Since then, more results, such as the existence of multiple geodesics or even infinitely many geodesics, have been studied extensively (c.f. \cite{Kl82, Gr}). Recently, a strong version was given by Colding and Minicozzi \cite{CMM08}, where they found good approximating sweep-outs, such that ``\emph{every curve in the sweepouts with length close to the longest must be close to a closed geodesic}" (see \cite{LW} for a proof using harmonic map flow).

%In this paper, we will develop a strong version of min-max methods for free boundary geodesics in the sense of Colding-Minicozzi, i.e. ``\emph{every curve in the tightened sweepouts with length close to the longest must be close to a free boundary geodesic}". The existence of a nontrivial free boundary geodesic is then a direct corollary.
%when the width is positive.
In this paper, we develop a version of min-max methods for free boundary geodesics with a strong approximation result in the sense of Colding-Minicozzi, i.e. ``\emph{every curve in the tightened sweepouts with length close to the longest must be close to a free boundary geodesic}". The existence of a nontrivial free boundary geodesic is then a direct corollary. Such strong property is shared by many min-max constructions of critical points for variational problems, c.f. \cite[\S 12.5]{AF65}\cite[\S 4.3]{P81}\cite[Proposition 3.1]{CD03}\cite[Theorem 1.9]{CMM08}\cite[Theorem 1.14]{CMR08}, and is very useful in other geometric problems, e.g. the proof of finite time extinction of 3-dimensional Ricci flow \cite{CMR08}. %Furthermore, the strong property makes our method potentially useful in other geometric problems.

The existence of geodesic with free boundary was discussed in various special cases before. Weinstein produced free boundary geodesic in a standard ball with Finsler metrics \cite[\S 4]{W}, by embedding the ball to a sphere. Nabutovsky and Rotman \cite{NR} studied the multiple solutions of geodesic loops where the constrained submanifold $N$ is a point. Compared to them, we extend the existence result to the full generality---the total space $M$ and the constraint $N$ can be any manifold and submanifold.
%Our result can be viewed as a generalization where the total space $M$ and the constraint $N$ can be any manifolds and submanifold.

Besides the case of geodesics, min-max methods have been studied widely in high dimensions \cite{P81, Jo89, F00, CD03, CMR08, MN12, L14}. Among them, the free boundary conditions were studied by Fraser \cite{F00} in the case of harmonic disks, and by Jost \cite{Jo89} and Li \cite{L14} in the case of embedded minimal disks.

\vspace{0.5em}
Let us introduce the notations and state the main results here. By the Nash embedding theorem, we can assume that $(M, g)$ is isometrically embedded in some Euclidean space $\R^{N}$. Denote $ds^2_0$ by the Euclidean metric. By scaling, we can assume that
\begin{itemize}
%\vspace{-5pt}
\addtolength{\itemsep}{-0.7em}
\setlength{\itemindent}{1em}

\item[(M1)] $\sup_{M}\|A_{M}\|\leq \frac{1}{16}$, $\sup_{N}\|A_{N}\|\leq \frac{1}{16}$, where $A_{M}$ and $A_{N}$ are the second fundamental forms of the embedding $M\subset\R^{N}$ and $N\subset\R^{N}$ respectively;

\item[(M2)] The injective radius of $M$ is at least $8$, and the curvature of $M$ is at most $\frac{1}{64}$;

\item[(M2)$^{\pr}$] The injective radius of $N$ is at least $4$, and the focal radius of $N$ is at least $4$;

\item[(M3)] For any $x, y\in M$, with $|x-y|_{ds^2_0}\leq 8$, $dist_{M}(x, y)\leq 2|x-y|$;

\item[(M3)$^{\pr}$] For any $x, y\in N$, with $|x-y|_{ds^2_0}\leq 8$, $dist_{N}(x, y)\leq 2|x-y|$.
\end{itemize}

\begin{remark}
If $M$ is compact, then the above constraints can be easily achieved by scaling. When $M$ is noncompact, we can assume that the conditions hold in a large convex domain containing $N$.
\end{remark}

Let $I=[0, 1]$ be the unit interval, and we will work in the Sobolev space $W^{1, 2}(I, M)$, where the $W^{1, 2}$-norm of a map $f: I\rightarrow M\subset\R^{N}$ is given by
$$\|f\|_{W^{1, 2}}^{2}=\int_{[0, 1]}|f(x)|^{2}+|f^{\pr}(x)|^{2}dx.$$
The energy of $f$ is defined by
$$E(f)=\int_{[0, 1]}|f^{\pr}(x)|^2dx.$$

Now we can define the total variational space as follows:
\begin{definition}\label{variational space1}
Let $\Om$ be the space of continuous mappings $\si: [0, 1]\times I\rightarrow M$, such that using coordinates $(t, x)\in [0, 1]\times I$,
\begin{itemize}
\vspace{-5pt}
\addtolength{\itemsep}{-0.7em}
\item $\si(t, \cdot)\in W^{1, 2}(I, M)$, and $\si(t, 0), \si(t, 1)\in N$, for all $t\in[0, 1]$;
\item $t\rightarrow \si(t, \cdot)$ is continuous as a map from $[0, 1]$ to $W^{1, 2}(I, M)$;
\item $\si(0, \cdot), \si(1, \cdot)$ are constant maps. 
\end{itemize}
\vspace{-5pt}
Each $\si$ is called a {\em sweep-out}.
\end{definition}
\begin{remark}
The notion of sweep-out comes from the special case when $(M, \partial M)$ is diffeomorphic to the unit disk $(D, \partial D)$, then we can choose $\si(t)$ to sweep out the whole disk.
\end{remark}

Given $\si_0\in\Om$, we let $[\si_0]$ to be the set of all $\si\in\Om$ which is homotopic to $\si_0$ in $\Om$. Then we can define the \emph{width} associated to $\si_0$ as:
\begin{equation}\label{width}
W=W\big([\si_0]\big)=\inf_{\si\in[\si_0]}\max_{t\in[0, 1]}E\big(\si(t)\big).
\end{equation}
The next result says that the width is always achieved by a geodesic of $M$ with free boundary on $N$.

\begin{theorem}\label{main theorem1}
Given $\si_0\in\Om$, with $W([\si_0])>0$, there exists a nontrivial geodesic $\ga: I\rightarrow M$ with free boundary on $N$, i.e. $\ga(0), \ga(1)\in N$, $\ga^{\pr}(0), \ga^{\pr}(1)\perp N$, and $E(\ga)=W_{[\si_0]}$.
\end{theorem}
\begin{remark}
When $N$ bounds a non-contractable disk in $M$,  we can find $\si_0$ with $W([\si_0])>0$ (see the discussion in the beginning of \S \ref{existence}).
\end{remark}

Theorem \ref{main theorem1} is a direct corollary of the following stronger theorem, which says that ``almost maximal implies almost critical".

\begin{theorem}\label{main theorem2}
Given $\si_0\in\Om$, with $W([\si_0])=W_0>0$, then there exists a sequence of sweep-outs $\{\ga_j\}_{j\in\N}\subset [\si_0]$, with
\vspace{-5pt}
$$\lim_{j\rightarrow\infty}\max_{t\in[0, 1]}E\big(\ga_j(t)\big)=W_0,$$
\vspace{-5pt}
such that, for any $\ep>0$, there exists a $\de>0$, and if $j>\frac{1}{\de}$, and if for some $t_0\in [0, 1]$
\begin{equation}\label{almost maximal}
E\big(\ga_j(t_0)\big)>W_0-\de,
\end{equation}
then $dist(\ga_j(t_0), G)<\ep$, where $G$ is the space of immersed geodesics with free boundary on $N$.
\end{theorem}
\begin{remark}
$\ga_j(t)$ will be piecewise geodesic by our construction, so $E\big(\ga_j(t_0)\big)=Length^2\big(\ga_j(t_0)\big)$, hence (\ref{almost maximal}) says that the length of $\ga_j(t_0)$ is almost maximal among the sweep-out $\ga_j(t)$.
\end{remark}

\vspace{0.5em}
The idea is to adapt the Birkhoff's curve shortening process (BCSP) to manifold with a constraint submanifold. Briefly, given a closed curve $\ga$ with $2L$ evenly spaced break points $\{x_0, x_1, \cdots, x_{2L}=x_0\}\subset S^1$, the BCSP first replaces each piece $\ga|_{[x_{2i}, x_{2i+2}]}$ on even intervals by a geodesic segment connecting $\ga(x_{2i})$ with $\ga(x_{2i+2})$, and then repeats the geodesic replacement process on odd intervals $[x_{2i-1}, x_{2i+1}]$. In our case, given a curve $\ga: [0, 1]\rightarrow M$, $\ga(0), \ga(1)\in N$, with $2L+1$ evenly spaced break points $\{x_0=0, x_1, \cdots, x_{2L}=1\}\subset [0, 1]$, we will first replace the boundary piece $\ga|_{[0, x_2]}$ (and $\ga|_{[x_{2L-2}, 1]}$) by geodesic segment $\ti{\ga}$ connecting $\ga(x_2)$ (and $\ga(x_{2L-2})$) to N, and then do replacements on inner pieces $\ga|_{[x_{2i}, x_{2i+2}]}$, $i\neq 0, L-1$ and $\ga|_{[x_{2i-1}, x_{2i+1}]}$ as BCSP.
%and then do geodesic replacement on $\ga|_{[x_{2i}, x_{2i+2}]}$, $i=1,\cdots, L-2$ and $\ga|_{[x_{2i-1}, x_{2i+1}]}$, $i=1, \cdots, L-1$ respectively as in BCSP.
Our modified BCSP satisfies properties analogous to BSCP (\S \ref{curve shortening}). Among them, one key ingredient is to show that the $W^{1, 2}$-norm difference $\|\ga|_{[0, x_2]}-\ti{\ga}\|_{1, 2}$ is controlled by the length difference $\Length(\ga|_{[0, x_2]})-\Length(\ti{\ga})$ (Lemma \ref{convexity1}). This is achieved by certain convexity estimates, which will be useful in other free boundary variational problems.

The paper is organized as follows. In \S \ref{preliminary}, we collect several preliminary results. In \S \ref{curve shortening}, we introduce the modified curve shortening process and several key properties. In \S \ref{existence}, we apply the curve shortening process to sweep-outs and finish the proof.

\vspace{5pt}
{\renewcommand\baselinestretch{1.0}\selectfont
\noindent\textbf{Acknowledgement:} {\small This paper is based upon work supported by the National Science Foundation under Grant No. 0932078 000, while the author was visiting the Mathematical Science Research Institute in Berkeley, California, during the Fall semester of 2013. The author would like to than Rick Schoen, Toby Colding and Bill Minicozzi for discussions.
}
\par}

%%%%%%%%%%%%%%%%%%%%%%%%%%%%%%%%%%%%%%%%%%%%%%%%%%%%%%%%%%%%%%%%%%%%%%%%%%%%%%%%%%%%%%%%%%%%%%%%%%%%%%%%%%%%%%%%%%%%%%%%%%

\section{Preliminary results}\label{preliminary}

Now we summarize several preliminary results in this section. The first fact is that the $W^{1, 2}$-norm bounds imply the H\"{o}lder continuity, i.e. given $x, y\in I$, using the Cauchy-Schwartz inequality,
\begin{equation}\label{Holder continuity}
|f(x)-f(y)|^{2}\leq\big(\int_{x}^{y}|f^{\pr}|\big)^{2}\leq |x-y|\int_{I}|f^{\pr}|^{2}.
\end{equation}
Thus $f$ is in $C^{\frac{1}{2}}$ if $f$ is in $W^{1, 2}$, and the $C^{\frac{1}{2}}$-norm is bounded by the $W^{1, 2}$-norm.

The second important result is the \emph{Wirtinger inequality} and the \emph{modified Wirtinger inequality}. Let us first state the Wirtinger inequality, which was used a lot in studying the existence of closed geodesics by Colding and Minicozzi \cite[page 165]{CM11}. Given $L>0$, let $f\in W^{1, 2}([0, L], \R)$, with $f(0)=f(L)=0$, then
$$\int_{0}^{L}\big(f(x)\big)^{2}dx\leq\frac{L^{2}}{\pi^{2}}\int_{0}^{L}|f^{\pr}(x)|^{2}dx.$$
For the modified Wirtinger inequality, we want similar estimates, while only assuming that the function is zero at one of the boundary points of $[0, L]$. Precisely, we have
\begin{lemma}\label{modified W inequality}
Let $f\in W^{1, 2}([0, L], \R)$, with $f(0)=0$, then
$$\int_{0}^{L}\big(f(x)\big)^{2}dx\leq\frac{L^{2}}{2}\int_{0}^{L}|f^{\pr}(x)|^{2}dx.$$
\end{lemma}
\begin{proof}
Using the fundamental theorem of calculus, $f(x)=\int_{0}^{x}f^{\pr}(s)ds$. By the Cauchy-Schwartz inequality,
$$|f(x)|^{2}\leq \big(\int_{0}^{x}|f^{\pr}(s)|ds\big)^{2}\leq x\int_{0}^{L}|f^{\pr}(s)|^{2}ds.$$
So
$$\int_{0}^{L}|f(x)|^{2}dx\leq \int_{0}^{L}x\big(\int_{0}^{L}|f^{\pr}(s)|^{2}ds\big)dx=\frac{L^{2}}{2}\int_{0}^{L}|f^{\pr}(s)|^{2}ds.$$
\end{proof}

The last fact is as follows. It appears in \cite[Lemma 5.2]{CM11}, and we need a version with better constants. It just comes out as a careful reproof of \cite[Lemma 5.2]{CM11}, and we give the details here for completeness.
\begin{lemma}\label{length estimate of normal component}
If $x, y\in M$, then $|(x-y)^{\perp}|\leq\frac{1}{8}|x-y|^{2}$, where $(x-y)^{\perp}$ is the normal component of $(x-y)$ to $M$ at $y$ in $\R^N$. Similar inequality also holds when $x, y\in N$, and when we take $(x-y)^{\perp}$ to be the normal component of $(x-y)$ to $N$ at $y$ in $\R^N$.
\end{lemma}
\begin{proof}
If $|x-y|\geq 8$, then it is trivially true. Assume now that $|x-y|\leq 8$. Choose $\al: [0, l]\rightarrow M$ as the minimizing unit speed geodesic from $y$ to $x$ in $M$ (which exists by (M2)), hence $l=dist_{M}(x, y)\leq 2|x-y|$ by (M3). Let $V$ be the unit normal vector
$$V=\frac{(x-y)^{\perp}}{|(x-y)^{\perp}|},$$
so $\lan\al^{\pr}(0), V\ran=0$. Hence
\begin{equation}\label{estimate of the perpendicular part}
\begin{split}
|(x-y)^{\perp}| & =\lan(x-y), V\ran=\int_{0}^{l}\lan\al^{\pr}(s), V\ran ds=\int_{0}^{l}\lan\al^{\pr}(0)+\int_{0}^{s}\al^{\pr\pr}(t)dt, V\ran ds,\\
                         & \leq \int_{0}^{l}\int_{0}^{s}|\al^{\pr\pr}(t)|dtds\leq\int_{0}^{l}\int_{0}^{s}|A_M(\al(t))||\al^{\pr}(t)|^2dtds,\\
                         & \leq \frac{1}{2}l^{2}\sup_{M}|A_M|\leq \frac{1}{8}|x-y|^{2}.
\end{split}
\end{equation}
The same proof works for $N$, as it satisfies the same properties as $M$ as embedded submanifolds of $\R^{N}$.
\end{proof}

%%%%%%%%%%%%%%%%%%%%%%%%%%%%%%%%%%%%%%%%%%%%%%%%%%%%%%%%%%%%%%%%%%%%%%%%%%%%%%%%%%%%%%%%%%%%%%%%%%%%%%%%%%%%%%%%%%%%%%%%%%

\section{Curve shortening process}\label{curve shortening}

Take $L\in\N$ to be a large integer.
%\begin{definition}\label{variational space2}
Let $\La$ be the space of piecewise linear maps\footnote{By linear map, we mean a (constant speed) geodesic.} $\ga: I\rightarrow M$ parametrized proportional to the arc-length with no more than $L-1$ break points, such that $\ga(0), \ga(1)\in N$, and each geodesic segment has length at most $1$, with Lipschitz bound $L$. Denote $G\subset\La$ to be the set of immersed geodesics with free boundary lying on $N$.
%\end{definition}
We will use the distance and topology on $\La$ given by the $W^{1, 2}$-norms on $W^{1, 2}(I, M)$.

In this section, we will construct the modified Birkhoff curve shortening map $\Psi: \La\rightarrow\La$, so that the following properties are satisfied\footnote{Similar map and properties appear in \cite[page 165]{CM11} in the case of closed curves.}:
\begin{itemize}
\vspace{-5pt}
\addtolength{\itemsep}{-0.7em}
\setlength{\itemindent}{1em}
\item[(1)] $\Psi(\ga)$ depends on $\ga$ continuously;
\item[(2)] $\Psi(\ga)$ is homotopic to $\ga$, and Length$\big(\Psi(\ga)\big)\leq$Length$(\ga)$;
\item[(3)] There is a continuous function $\phi: [0, \infty)\rightarrow [0, \infty)$, with $\phi(0)=0$, such that,
$$dist^{2}\big(\ga, \Psi(\ga)\big)\leq\phi\big(\frac{\textrm{Length}^{2}(\ga)-\textrm{Length}^{2}\big(\Psi(\ga)\big)}{\textrm{Length}^{2}\big(\Psi(\ga)\big)}\big);$$
%\item[(4)] $\Psi(\ga)$ is perpendicular to $N$ at boundary points;
\item[(4)] If $\Psi(\ga)=\ga$, then $\ga\in G$, i.e. fixed points of $\Psi$ are immersed geodesics with free boundary lying on $N$; 
\item[(5)] Given $\ep>0$, there exists $\de>0$, such that if $\ga\in\La$, and $dist(\ga, G)\geq\ep$, then
$$\textrm{Length}\big(\Psi(\ga)\big)\leq \textrm{Length}(\ga)-\de.$$
\end{itemize}

%%%%%%%%%%%%%%%%%%%%%%%%%%%%%%%%%%%%%%%%%%%%%%%%%%%%%%%%%%%

\subsection{Defining $\Psi$.}\label{defining Psi}
Fix a partition of $I=[0, 1]$ by choosing $2L-1$ successive evenly spaced break points:
$$x_{0}=0, x_{1}, x_{2}, \cdots, x_{2L}=1\in [0, 1],$$
such that $|x_{j+1}-x_{j}|=\frac{1}{2L}$. Similarly to that in \cite{CM11}, $\Psi$ is defined by four steps:
\begin{itemize}
\vspace{-5pt}
\addtolength{\itemsep}{-0.7em}
\setlength{\itemindent}{2em}
\item[\textbf{Step 1.}] Replace $\ga$ on the boundary even intervals $[0, x_2]$ and $[x_{2L-2}, 0]$ by the minimizing geodesics from $\ga(x_2)$ and $\ga(x_{2L-2})$ to $N$ respectively\footnote{Such geodesics exist as $\ga(x_2)$ and $\ga(x_{2L-2})$ lie within the boundary cut locus by (M2)$^\pr$.}, and then replace $\ga$ on each inner even interval $[x_{2j}, x_{2j+2}]$ by the linear map with the same endpoints to get a piecewise linear map $\ga_e: [0, 1]\rightarrow M$.
%such that $\ga_e(0), \ga(1)$ lie in $N$ and meets $N$ perpendicularly. 
%Replace $\ga|_{[0, x_2]}$ and $\ga|_{[x_{2L-2}, 1]}$ by the minimizing geodesics $\ti{\ga}_e|_{[0, x_2]}$ from $\ga(x_2)$ to $N$ and the minimizing geodesic $\ti{\ga}_e|_{[x_{2L-2}, 1]}$ from $\ga(x_{2L-2})$ to $N$ respectively (such geodesics exist as $\ga(x_2)$ and $\ga(x_{2L-2})$ lie within the boundary cutlocus), and then replace $\ga|_{[x_{2j}, x_{2j+2}]}$ by the minimizing geodesic segment $\ti{\ga}_e|_{[x_{2j}, x_{2j+2}]}$ from $\ga(x_{2j})$ to $\ga(x_{2j+2})$, for $j=1, \cdots, L-2$. Denote the new map by $\ti{\ga}_e: [0, 1]\rightarrow M$, then $\ti{\ga}_e(0)$ and $\ti{\ga}_e(1)$ lie on $N$ and $\ti{\ga}_e$ meets $N$ perpendicularly there.
%If $\ga|_{[0, x_{2}]}$ is non-constant, then replace $\ga|_{[0, x_{2}]}$ by the minimizing geodesic from $\ga(x_{2})$ to $N$; while if $\ga|_{[0, x_{2}]}$ is constant, then inductively find the first $x_{2j}$, such that $\ga_{[0, x_{2j}]}$ is non-constant, and replace $\ga|_{[0, x_{2j}]}$ by the minimizing geodesic from $\ga(x_{2j})$ to $N$. We also do similar thing to the other end of $[0, 1]$ from $x_{2L}=1$.
\item[\textbf{Step 2.}] Reparametrize $\ga_e$ to get a constant speed curve $\ti{\ga}_e$, i.e. $\ti{\ga}_e$ is parametrized proportional to the arc length. 
%Reparametrize $\ti{\ga}_e$ on $[0, 1]$ to $\ga_e$ so that $\ga_e$ is parametrized proportional to the arc length.
\item[\textbf{Step 3.}] Denote $\ti{x}_j$ to be the image of $x_j$ under this reparametrization, i.e. $\ga_e(x_j)=\ti{\ga}_e(\ti{x}_j)$. Replace $\ti{\ga}_e$ on each odd interval $[\ti{x}_{2j-1}, \ti{x}_{2j+1}]$ by the linear map with the same endpoints to get a piecewise linear map $\ga_o: [0, 1]\rightarrow M$.
%Denote $\bar{x}_{2j-1}$ to be the image of $x_{2j-1}$ under the reparametrization, i.e. $\ga_e(\bar{x}_{2j-1})=\ti{\ga}_e(x_{2j-1})$ for $j=1,\cdots, L-1$. Replace $\ga_e|_{[\bar{x}_{2j-1}, \bar{x}_{2j+1}]}$ by the minimizing geodesic segment $\ti{\ga}_o|_{[\bar{x}_{2j-1}, \bar{x}_{2j+1}]}$ from $\ga_e(\bar{x}_{2j-1})$ to $\ga_e(\bar{x}_{2j+1})$. Denote the new map by $\ti{\ga}_o$.
\item[\textbf{Step 4.}] Reparametrize $\ga_o$ to get a constant speed curve $\ti{\ga}_o$, which is then $\Psi(\ga)$.
%Reparametrize $\ti{\ga}_o$ on $[0, 1]$ to $\ga_o$ so that $\ga_o$ is parametrized proportional to the arc length.  
\end{itemize}

In fact, each of the steps is energy non-increasing. The linear replacements obviously reduce the energy, as the linear maps (with both fixed endpoints, or with one endpoint fixed and the other free on $N$) minimize energy. The reparametrizations reduce the energy because of the Cauchy-Schwartz inequality, since for a map $\ga: [0, 1]\rightarrow M$,
$${\rm Length}^2(\ga)\leq E(\ga),$$
where equality holds if and only if $\ga$ has constant speed almost everywhere. Now we will prove these properties of $\Psi$ in the following.

%%%%%%%%%%%%%%%%%%%%%%%%%%%%%%%%%%%%%%%%%%%%%%%%%%%%%%%%%%%

\subsection{Property (3) of $\Psi$.}

Using the triangle inequality, we only need to bound $dist(\ga, \ga_e)$ and $dist(\ga_e, \ti{\ga}_e)$ by
%$\big(\Length^2(\ga)-\Length^2(\ga_e)\big)$ and $\big(\Length^2(\ga_e)-\Length^2(\ti{\ga}_e)\big)$ respectively
the the difference of their length square (as well as those in Step 3 and Step 4). In \cite{CMM08, CM11}, Colding and Minicozzi showed that the $W^{1, 2}$-distance of two curves $\si, \ti{\si}$ can be bounded by the difference of their length square when $\ti{\si}$ is a linear map with the same endpoints as $\si$, and they also showed analogous bound for the reparametrization as Step 2 and Step 4. Hence the bounds of $dist(\ti{\ga}_e, \ga_o)$ and $dist(\ga_o, \ti{\ga}_o)$ are almost the same as those in \cite[\S 3.2]{CM11}, except that the parameter space is $[0, 1]$ instead of $S^1$.
%after changing the domain from $S^1$ to $[0, 1]$.
The only estimate left in our case is to find similar bound of $dist(\si, \ti{\si})$ when $\si$ has one endpoint on $N$ and $\ti{\si}$ is the minimizing geodesic from the other endpoint of $\si$ to $N$. Particularly,
\begin{lemma}\label{convexity0}
(\cite[Lemma 5.1]{CM11}) There exists $C$ so that if $I$ is an interval of length at most $\frac{1}{L}$, $\si_1: I\rightarrow M$ is a Lipschitz curve with $|\si_1^{\pr}|\leq L$, and $\si_2: I\rightarrow M$ is the minimizing geodesic with the same endpoints, then
\begin{equation}\label{convexity inequality0}
dist^2(\si_1, \si_2)\leq C\big(E(\si_1)-E(\si_2)\big).
\end{equation}
\end{lemma}

We have the analog for the free boundary case.
\begin{lemma}\label{convexity1}
There exists $C^{\pr}$ so that if $I$ is an interval of length at most $\frac{1}{L}$, say $I=[0, l]$, $l\leq\frac{1}{L}$, $\si_{1}: I\rightarrow (M, N)$ is a Lipschitz curve with $|\si_1^{\pr}|\leq L$, with one endpoint lying on $N$, i.e. $\si_{1}(l)\in N$, and $\si_{2}: I\rightarrow (M, N)$ is the minimizing geodesic from $\si_{2}(0)$ to $N$, i.e. $\si_2(0)=\si_{1}(0)$, $\si_2(l)\in N$ and $\si_2^{\pr}(l)\perp N$, then
\begin{equation}\label{convexity inequality1}
dist^2(\si_1, \si_2)\leq C^{\pr}\big(E(\si_1)-E(\si_2)\big).
\end{equation}
\end{lemma}
\begin{proof}
Integration by part and using that $\si_1$ and $\si_2$ are equal at $0$,
\begin{equation}\label{integration by part convexity}
\begin{split}
\int_I |\si_1^{\pr}|^2-\int_I |\si_2^{\pr}|^2 & -\int_I |\si_1^{\pr}-\si_2^{\pr}|^2=2\int_I\lan\si_2^{\pr}, (\si_1-\si_2)^{\pr}\ran ds\\
                                                                      & =2\lan\si_2^{\pr}, (\si_1-\si_2)\ran|_{s=l}-2\int_I\lan(\si_1-\si_2), \si_2^{\pr\pr}\ran ds.
\end{split}
\end{equation}

For the first term in the last line, as $\si_2$ is a minimizing geodesic, with length less or equal than that of $\si_1$, $|\si_2^{\pr}|\leq L$. Also by the fundamental theorem of calculus, the Cauchy-Schwartz inequality, and the fact that $\si_1(0)=\si_2(0)$,
\begin{equation}
\begin{split}
|(\si_1-\si_2)(l)| & =|\int_{0}^{l}(\si_1-\si_2)^{\pr}ds|\leq \int_{0}^{l}|(\si_1-\si_2)^{\pr}|ds\\
                                                 & \leq l^{\frac{1}{2}}\big(\int_{0}^{l}|(\si_1-\si_2)^{\pr}|^{2}ds\big)^{\frac{1}{2}}.
\end{split}
\end{equation}
Since $\si_2^{\pr}(l)$ is normal to $N$, we can use Lemma \ref{length estimate of normal component} for $x=\si_1(l)\in N$ and $y=\si_{2}(l)\in N$, hence
\begin{equation}\label{convexity 2}
\begin{split}
|\lan\si_2^{\pr}, (\si_1-\si_2)\ran|_{s=l}| &\leq \frac{1}{8}|\si_2^{\pr}(l)|\cdot |(\si_1-\si_2)(l)|^{2}\leq \frac{1}{8}L\cdot l\big(\int_{0}^{l}|(\si_1-\si_2)^{\pr}|^{2}ds\big)\\
                                                       &\leq\frac{1}{8}\big(\int_I |\si_1^{\pr}-\si_2^{\pr}|^{2}ds\big).
\end{split}
\end{equation}

For the second term in the last line of (\ref{integration by part convexity}), we can use similar argument as in \cite{CMM08, CM11}. Using the geodesic equation of $\si_2$ in $M$, i.e. $\si_2^{\pr\pr}=A_M(\si_2^{\pr}, \si_2^{\pr})$, where $A_M$ is the second fundamental form of the embedding of $M$ in $\R^N$, and by (M1) in the Section \ref{introduction},
$$|\si_2^{\pr\pr}|\leq(\sup_M|A_M|)|\si_2^{\pr}|^2\leq \frac{1}{16}|\si_2^{\pr}|^2\leq\frac{1}{16}L^2.$$
Since $\si_2^{\pr\pr}$ is normal to $M$, we can use Lemma \ref{length estimate of normal component} for $x=\si_1(s)\in M$ and $y=\si_2(s)\in M$, $s\in I$, hence
\begin{equation}\label{convexity 3}
\begin{split}
|\int_I\lan(\si_1-\si_2), \si_2^{\pr\pr}\ran ds| & \leq \int_I \frac{1}{8} |\si_2^{\pr\pr}|\cdot |\si_1-\si_2|^2ds \leq \frac{1}{8}\cdot\frac{L^2}{16}\int_I |\si_1-\si_2|^2ds\\
                                                              & \leq \frac{1}{128}L^2\cdot\frac{l^2}{2}\int_I |\si_1^{\pr}-\si_2^{\pr}|^2ds\leq \frac{1}{8}\int_I |\si_1^{\pr}-\si_2^{\pr}|^2ds,
\end{split}
\end{equation}
where we used the modified Wirtinger inequality (c.f. Lemma \ref{modified W inequality}) in the third $``\leq"$.

Now plug (\ref{convexity 2})(\ref{convexity 3}) into (\ref{integration by part convexity}),
$$\int_I |\si_1^{\pr}-\si_2^{\pr}|^2ds\leq 2\big(\int_I |\si_1^{\pr}|^2ds- \int_I |\si_2^{\pr}|^2ds\big).$$
Applying the modified Wirtinger inequality (c.f. Lemma \ref{modified W inequality}) with the above inequality, we can get (\ref{convexity inequality1}).
\end{proof}

Now we can finish the proof of Property (3).
\begin{proof}
(of Property (3) of $\Psi$) For Step 1 of $\Psi$, apply Lemma \ref{convexity1} on the boundary intervals $[0, x_2]$ and $[x_{2L-2}, 1]$, and use Lemma \ref{convexity0} on the inner intervals $[x_{2j}, x_{2j+2}]$, $j=1,\cdots, L-2$, and sum them together, we have
\begin{equation}\label{first W12 bound}
dist^2(\ga, \ga_e)\leq C\big(E(\ga)-E(\ga_e)\big)\leq C^{\pr}\big(\frac{Length^2(\ga)-Length^2(\ga_e)}{Length^2(\ga_e)}\big),
\end{equation}
where in the last $``\leq"$ we used the fact that $E(\ga)=Length^2(\ga)$ as $\ga$ has constant speed, and $Length^2(\ga_e)\leq E(\ga_e)$, and $Length(\ga_e)\leq Length(\ga)\leq L$.

The bound of $dist(\ga_e, \ti{\ga}_e)$ is similar to that in \cite{CMM08, CM11}, but as different parametrization is used here, we will give the details for completeness. In fact, $\ga_e$ can be viewed as a reparametrization of $\ti{\ga}_e$, i.e. $\ga_e=\ti{\ga}_e\circ P$, where $P: I\rightarrow I$ is a monotone piecewise linear map. Since $\ti{\ga}_e$ has constant speed, $|\ti{\ga}_e^{\pr}|=Length(\ti{\ga}_e)/1=Length(\ti{\ga}_e)$. Also $\int_I P^{\pr}ds=1$, hence
\begin{equation}\label{estimate of the reparametrization0}
\begin{split}
\int_I (P^{\pr}-1)^2ds & =\int_I |P^{\pr}|^2ds -1 = \int_I\big(\frac{|\ga_e^{\pr}|}{|\ti{\ga}^{\pr}_e\circ P|}\big)^2ds-1\\
                              & =\frac{1}{Length^2(\ti{\ga}_e)}\int_I |\ga_e^{\pr}|^2ds-1=\frac{E(\ga_e)-Length^2(\ti{\ga}_e)}{Length^2(\ti{\ga}_e)},\\
                              & \leq \frac{Length^2(\ga)-Length^2(\ti{\ga}_e)}{Length^2(\ti{\ga}_e)},
\end{split}
\end{equation}
where in the last $``\leq"$ we used the fact that $E(\ga_e)\leq E(\ga)=Length^2(\ga)$ as $\ga$ has constant speed.

To bound $dist(\ga_e, \ti{\ga}_e)$, as $\ga_e(0)=\ti{\ga}_e(0)$ and $\ga_e(1)=\ti{\ga}_e(1)$, we can combine the Wirtinger inequality with the following estimate,
\begin{equation}\label{estimate of the reparametrization1}
\begin{split}
\int_I |\ga_e^{\pr}-\ti{\ga}_e^{\pr}|^2ds & = \int_I |(\ti{\ga}_e^{\pr}\circ P)P^{\pr}-\ti{\ga}_e^{\pr}|^2ds\\
                                                      & \leq 2\int_I |(\ti{\ga_e}^{\pr}\circ P)P^{\pr}-\ti{\ga}_e^{\pr}\circ P|^2ds+2\int_I |\ti{\ga}_e^{\pr}\circ P-\ti{\ga}_e^{\pr}|^2ds
\end{split}
\end{equation}
For the first term, using the fact that $\ti{\ga}_e$ has constant speed, we have $|\ti{\ga}_e^{\pr}|\leq L$, hence
\begin{equation}\label{estimate of the reparametrization2}
\int_I |(\ti{\ga}_e^{\pr}\circ P)P^{\pr}-\ti{\ga}_e^{\pr}\circ P|^2ds\leq\sup_I |\ti{\ga}_e^{\pr}|^2\cdot\int_I |P^{\pr}-1|^2ds\leq L^2\int_I |P^{\pr}-1|^2ds.
\end{equation}
For the second term, using the fact that $\ti{\ga}_e$ is a piecewise linear map, $|\ti{\ga}_e^{\pr\pr}|=|A_M(\ti{\ga}_e^{\pr}, \ti{\ga}_e^{\pr})|\leq\frac{L^2}{16}$ (using property (M1) in Section \ref{introduction}) away from break points. So if $x, y\in I$ are not separated by a break point of $\ti{\ga}_e$, then by the fundamental theorem of calculus,
\begin{equation}\label{estimate of derivative of geodesic}
|\ti{\ga}_e^{\pr}(x)-\ti{\ga}_e^{\pr}(y)|\leq \frac{L^2}{16}|x-y|.
\end{equation}
Now divide $[0, 1]$ into two sets $I_1$ and $I_2$, where $I_1$ is the set of points within distance $(\int_I |P^{\pr}-1|^2)^{1/2}$ of a break point for $\ti{\ga}_e$, so $\Length(I_2)=L\cdot (2(\int_I |P^{\pr}-1|^2)^{1/2})=2L(\int_I |P^{\pr}-1|^2)^{1/2}$. As $P(0)=0$, we have
$$|P(x)-x|\leq |\int_0^x (P^{\pr}(s)-1)ds|\leq x^{1/2}\big(\int_I |P^{\pr}-1|^2\big)^{1/2}\leq \big(\int_I |P^{\pr}-1|^2\big)^{1/2}.$$
So if $x\in I_2$, then $P(x)$ lies in the same geodesic segment as $x$. Hence using (\ref{estimate of derivative of geodesic}) and the Wirtinger inequality,
$$\int_{I_2}|\ti{\ga}_e^{\pr}\circ P-\ti{\ga}_e^{\pr}|^2ds\leq\frac{L^4}{256}\int_{I_2}|P(s)-s|^2ds\leq\frac{L^4}{256}\int_I |P^{\pr}-1|^2ds.$$
Also as $|\ti{\ga}_e^{\pr}|\leq L$, so on $I_1$,
$$\int_{I_1}|\ti{\ga}_e^{\pr}\circ P-\ti{\ga}_e^{\pr}|^2ds\leq 4L^2 \Length(I_2)\leq 8L^3(\int_I |P^{\pr}-1|^2)^{1/2}.$$
Combining the above two inequalities with (\ref{estimate of the reparametrization2})(\ref{estimate of the reparametrization1}) and (\ref{estimate of the reparametrization0}), together with (\ref{first W12 bound}), we can prove property (3) for Step 1 and Step 2 by realizing that $Length(\ga_e), Length(\ti{\ga}_e)\geq Length(\Psi(\ga))$. Estimates for $dist(\ti{\ga}_e, \ga_o)$ and $dist(\ga_o, \ti{\ga}_o)$ in Step 3 and Step 4 are the same and even easier as we only need Lemma \ref{convexity0}, so we omit the proof here.
\end{proof}

%%%%%%%%%%%%%%%%%%%%%%%%%%%%%%%%%%%%%%%%%%%%%%%%%%%%%%%%%%%

\subsection{Property (4) of $\Psi$.} The fact that fixed points of $\Psi$ are geodesics in the case of closed curve was discussed in \cite[page 169]{CM11}\cite[\S 2]{Cr88}.

\begin{lemma}\label{fixed point of Psi}
Given $\ga\in\La$, if $\Psi(\ga)=\ga$, then $\ga\in G$, i.e. $\ga$ is a geodesic with free boundary on $N$.
\end{lemma}
\begin{proof}
We can assume that $\ga$ is not a point curve, or the statement is trivial. By the discussion at the end of \S \ref{defining Psi}, the energy is non-increasing under the four steps, i.e. $E(\ga)\geq E(\ga_e)\geq E(\ti{\ga}_e)\geq E(\ga_o)\geq E(\Psi(\ga))$. As $E(\Psi(\ga))=E(\ga)$, the energy must be the same. For $\ga_e\rightarrow\ti{\ga}_e$ and $\ga_o\rightarrow\ti{\ga}_o=\Psi(\ga)$, by (\ref{estimate of the reparametrization0}) where $E(\ga_e)=E(\ti{\ga}_e)=Length^2(\ti{\ga}_e)$ (similar estimates also hold for $\ga_o$ and $\ti{\ga}_o$), we know that $\ga_e=\ti{\ga}_e$ and $\ga_o=\Psi(\ga)=\ga$, i.e. the reparametrization $P\equiv id$. As $E(\ga)=E(\ga_e)=E(\ga_o)$, by Lemma \ref{convexity0} and Lemma \ref{convexity1}, we know that $\ga=\ga_e=\ga_o$. The fact $\ga=\ga_e$ implies that the break points of $\ga$ can only appear at
%$\ga$ has at most $L-1$ break points at
$\ga(x_{2j})$, $j=1, \cdots, L-1$, and $\ga$ is perpendicular to $N$ at boundary points, i.e. $\ga$ has free boundary on $N$. Also as the points $x_{2j}$, $j=1, \cdots, L-1$ are smooth points of $\ga_o$, hence $\ga$ has no break points.
\end{proof} 

%%%%%%%%%%%%%%%%%%%%%%%%%%%%%%%%%%%%%%%%%%%%%%%%%%%%%%%%%%%

\subsection{Property (5) of $\Psi$.} This is a direct corollary of other properties of $\Psi$ by a contradiction argument. Similar argument in the case of closed curves appeared in \cite[page 169]{CM11}. Suppose in contradiction, say there exists an $\ep>0$, and a sequence of $\ga_j\in\La$, such that
$$Length(\Psi(\ga_j))\geq Length(\ga_j)-\frac{1}{j},\qquad dist(\ga_j, G)\geq\ep.$$
As the second condition implies that $Length(\ga_j)$ is bounded away from zero, the first condition and property (3) implies that
$$dist(\ga_j, \Psi(\ga_j))\rightarrow 0.$$
Note that the space $\La$ is compact, as every $\si\in\La$ depends continuously on the boundary points on $N$ and the $L-1$ break points in the compact manifold $M$ (or in a compact convex region when $M$ is non-compact). Hence a subsequence of $\{\ga_j\}$ will converge to some $\ga\in\La$. As $\Psi$ is continuous on $\La$, i.e. property (1), $\Psi(\ga)=\ga$, which implies that $\ga\in G$ by Lemma \ref{fixed point of Psi}, contradiction to that $dist(\ga, G)=\lim_{j\rightarrow\infty}dist(\ga_j, G)\geq\ep$.

%%%%%%%%%%%%%%%%%%%%%%%%%%%%%%%%%%%%%%%%%%%%%%%%%%%%%%%%%%%

\subsection{Property (1) of $\Psi$.} Now we prove the continuity of $\Psi$. Similar argument in the case of closed curves appeared in \cite{CM11}. Here we need to carefully deal with the case of free boundary linear replacement. The continuity of the first two steps for $\Psi$ follows from the following lemma, and the continuity of the last two steps follows from the closed case (c.f. \cite[Lemma 5.3]{CM11}).

\begin{lemma}\label{continuity of Psi1}
Given a large integer $L$. Let $\ga: [0, 1]\rightarrow M$ be a $W^{1, 2}$-map, with $\ga(0), \ga(1)\in N$, $E(\ga)\leq L$. If $\ga_e$, $\ti{\ga}_e$ are given by the first two steps for $\Psi$ in \S \ref{defining Psi}, then the map $\ga\rightarrow\ti{\ga}_e$ is continuous from $W^{1, 2}(I, M)$ to $\La$.
\end{lemma}
Recall that $\{x_0, x_2 \cdots, x_{2L-2}, x_{2L}=1\}$ are the evenly spaced points where we do linear replacement. By (\ref{Holder continuity}), $|\ga(x_{2j})-\ga(x_{2j+2})|\leq\big(\frac{1}{L}E(\ga)\big)^{1/2}\leq 1$, hence $d_M(\ga(x_{2j}), \ga(x_{2j+2}))\leq 2$ by (M3) in \S \ref{introduction}, and we can apply Step 1 of \S \ref{defining Psi} by (M2) and (M2)$^{\pr}$ in \S \ref{introduction}. Now recall two observations used in \cite{CM11} ((C1)(C2) already appeared in \cite[page 170]{CM11}), 
\begin{itemize}
%\vspace{-5pt}
\addtolength{\itemsep}{-0.7em}
\setlength{\itemindent}{1em}

\item[(C1)] Curves which are $W^{1, 2}$ close are also $C^0$ close , hence the points $\ga_e(x_{2j})=\ga(x_{2j})$ are continuous with respect to $W^{1, 2}$-norm of $\ga$.

\item[(C2)] Let $\Ga=\{(x, y)\in M\times M:\ dist_M(x, y)\leq 4\}$, and define
\begin{equation}\label{definition of H}
H: \Ga\rightarrow C^{1}([0, \frac{1}{L}], M)
\end{equation}
such that $H(x, y)$ is the linear map from $x$ to $y$, then $H$ is continuous on $\Ga$.

\item[(C2)$^{\pr}$] Let $\Ga^{\pr}=\{x\in M:\ dist_M(x, N)\leq 2\}$, and define
\begin{equation}\label{definition of H prime}
H^{\pr}: \Ga^{\pr}\rightarrow C^1([0, \frac{1}{L}], M)
\end{equation}
such that $H^{\pr}(x)$ is the minimizing geodesic from $x$ to $N$, then $H^{\pr}$ is continuous on $\Ga^{\pr}$.
\end{itemize}

To make (C2)$^{\pr}$ precise, we have the following lemma.
\begin{lemma}\label{boundary continuity in free boundary case}
Given $L\geq 2$, and $x_1, x_2\in \Ga^{\pr}$, and let $\si_1, \si_2: [0, \frac{1}{L}]\rightarrow M$ be two minimizing geodesic from $\si_1(0)=x_1$ and $\si_2(0)=x_2$ to $N$ respectively, with $\si_1(\frac{1}{L}), \si_2(\frac{1}{L})\in N$, and meet $N$ orthogonally there, then there exists a continuous function $\phi: [0, \infty)\rightarrow [0, \infty)$, with $\phi(0)=0$, such that
$$dist(\si_1, \si_2)\leq  \phi\big(d_M(x_1, x_2)\big).$$
\end{lemma}
\begin{proof}
Using integration by parts as in Lemma \ref{convexity1}
\begin{equation}\label{integration by part2}
\begin{split}
\int |\si_1^{\pr}|^2 & -\int |\si_2^{\pr}|^2-\int |\si_1^{\pr}-\si_2^{\pr}|^2=2\int \lan\si_2^{\pr}, (\si_1-\si_2)^{\pr}\ran dt\\
                         & =2\underbrace{\lan\si_2^{\pr}(\frac{1}{L}), \si_1(\frac{1}{L})-\si_2(\frac{1}{L})\ran}_{I_1}-2\underbrace{\lan\si_2^{\pr}(0), \si_1(0)-\si_2(0)\ran}_{I_2}-2\underbrace{\int \lan(\si_1-\si_2), \si_2^{\pr\pr}\ran}_{I_3}.
\end{split}
\end{equation}
First using the fact that $|\si_2^{\pr}|=\Length(\si_2)=d_M(x_2, N)\leq 2$,
%$$|I_2|\leq\frac{1}{16}|\si_2^{\pr}(0)| d_M^2(\si_1(0), \si_2(0))\leq\frac{1}{16}d_M(x_2, N)d_M^2(x_1, x_2).$$
$$|I_2|\leq|\si_2^{\pr}(0)|\cdot |(\si_1-\si_2)(0)|\leq 2 d_M(x_1, x_2).$$
Also using Lemma \ref{length estimate of normal component} and the fundamental theorem of calculus,
\begin{displaymath}
\begin{split}
|I_1| & \leq \frac{1}{8}|\si_2^{\pr}(\frac{1}{L})|\cdot|\si_1(\frac{1}{L})-\si_2(\frac{1}{L})|^2\\
       & \leq \frac{1}{8} d_M(x_2, N)\cdot \big|\si_1(0)-\si_2(0)+\int_0^{1/L}\big(\si_1^{\pr}-\si_2^{\pr}\big)ds\big|^2\\
       & \leq \frac{1}{8} d_M(x_2, N)\cdot  2\big[|x_1-x_2|^2+|\int_0^{1/L} (\si_1^{\pr}-\si_2^{\pr})ds|^2\big]\\
       & \leq \frac{1}{8} d_M(x_2, N)\cdot  2\big[|x_1-x_2|^2+\frac{1}{L}\int_0^{1/L} |\si_1^{\pr}-\si_2^{\pr}|^2ds\big]\\
       & \leq \frac{1}{2} d_M^2(x_1, x_2)+\frac{1}{4}\int_0^{1/L}|\si_1^{\pr}-\si_2^{\pr}|^2ds,
\end{split}
\end{displaymath}
where we used the Cauchy-Schwartz inequality in the fourth $``\leq"$, and the fact that $d_M(x_2, N)\leq 2$ and $L\geq 2$ in the last $``\leq"$.

Similarly to (\ref{convexity 3}), %and using the fact that $|\si_2^{\pr}|\leq d_M(x_2, N)\leq 2$,
%$$|I_3|\leq \frac{1}{8}\int |\si_1^{\pr}-\si_2^{\pr}|^2.$$
\begin{displaymath}
\begin{split}
|I_3| &\leq \int_0^{1/L} |\lan\si_1-\si_2, \si_2^{\pr\pr}\ran|dt \leq \int_0^{1/L}\frac{1}{8}|\si_2^{\pr\pr}|\cdot |\si_1-\si_2|^2dt\\
        & \leq \frac{1}{8}(\sup_M |A_M|)\cdot |\si_2^{\pr}|^2 \int_0^{1/L} |\si_1-\si_2|^2dt\\
        & \leq \frac{1}{32} \int_0^{1/L} \big|(\si_1-\si_2)(0)+\int_0^t(\si_1^{\pr}-\si_2^{\pr})(s)ds \big|^2dt\\
        & \leq \frac{1}{16} \int_0^{1/L} \big[|\si_1(0)-\si_2(0)|^2+t\int_0^t|\si_1^{\pr}-\si_2^{\pr}|^2(s)ds\big]dt\\
        & \leq \frac{1}{16 L} d_M^2(x_1, x_2)+\frac{1}{32 L}\int_0^{1/L} |\si_1^{\pr}-\si_2^{\pr}|^2(t)dt.
\end{split}
\end{displaymath}

By plugging the above to (\ref{integration by part2}), we get
$$E(\si_1)-E(\si_2)+(1+\frac{1}{8 L})d_M^2(x_1, x_2)+4 d_M(x_1, x_2)\geq c\int|\si_1^{\pr}-\si_2^{\pr}|^2ds,$$
for some $c>0$. Reverse the role of $\si_1$ and $\si_2$, and sum the above inequality together, we get
$$\varphi\big(d_M(x_1, x_2)\big) \geq c\int|\si_1^{\pr}-\si_2^{\pr}|^2ds,$$
for $\varphi(x)=(1+\frac{1}{8 L})x^2+4x$.

Similarly by the fundamental theorem of calculus,
\begin{displaymath}
\begin{split}
\int|\si_1-\si_2|^2 & \leq \int\Big|\big[\si_1(0)-\si_2(0)+\int_0^s(\si_1^{\pr}-\si_2^{\pr})dt\big]\Big|^2ds\\
                          &\leq \frac{2}{L}d_M^2(x_1, x_2)+\frac{1}{L}\int|\si_1^{\pr}-\si_2^{\pr}|^2.
\end{split}
\end{displaymath}
\end{proof}
\begin{remark}\label{remark of boundary continuity}
The lemma is still true if the defining intervals of $\si$ has length less than $\frac{1}{L}\leq\frac{1}{2}$.
To get (C2)$^{\pr}$, we can use the above lemma and the fundamental theorem of calculus to show that $d_M(\si_1(\frac{1}{L}), \si_2(\frac{1}{L}))$ depends continuously on $d_M(\si_1(0), \si_2(0))$, and then use (C2). 
\end{remark}

\vspace{5pt}
\begin{proof}
(of Lemma \ref{continuity of Psi1}) %It follows from (C1) (C2) (C2)' that $\ga\rightarrow\ga_e$ is continuous, as $\ga_e(x_{2j})=\ga(x_{2j})$ are continuous with respect to $\ga$ by (C1), and the geodesic segments $\ga_e|_{[x_{2j}, x_{2j+2}]}$ are continuous with respect to $\ga_e(x_{sj})$ by (C2)(C2)'. 
Given two $W^{1, 2}$-maps $\ga^1$, $\ga^2$ as in the lemma which are $W^{1, 2}$ close, we can assume that $\ga^1$ is not a constant curve, or the proof is trivial. Let $a^i_0=d_M(\ga^i(x_2), N)$, $a^i_{L-1}=d_M(\ga^i(x_{2L-2}), N)$, and $a^i_j=d_M\big(\ga^i(x_{2j}), \ga^i(x_{2j+2})\big)$, $j=1, \cdots, L-2$, and $S^i=\sum_{j=0}^{L-1}a^i_j$, $i=1, 2$, then $S^i>0$. By (C1), the points $\ga^i(x_{2j})=\ga_e^i(x_{2j})$, $j=1, \cdots, L-1$ and hence the numbers $a^i_j$, $S^i$ are all continuous with respect to $\ga^i$ for $i=1, 2$. Therefore, the geodesic segments $\ga_e^1|_{[x_{2j}, x_{2j+2}]}$ and $\ga_e^2|_{[x_{2j}, x_{2j+2}]}$ are $C^1$ close on $[x_{2j}, x_{2j+2}]$ for $j=0, 1,\cdots, L-1$ by (C2) and (C2)$^{\pr}$, so $\ga\rightarrow\ga_e$ is continuous.

Since the reparametrization from $\ga_e\rightarrow\ti{\ga}_e$ fixes the boundary point, $\ti{\ga}_e^1(0)$ and $\ti{\ga}_e^2(0)$ are close by Remark \ref{remark of boundary continuity}. To show that $\ti{\ga}_e^1$ and $\ti{\ga}_e^2$ are $W^{1, 2}$ close, we only need to show that $\int|(\ti{\ga}_e^1-\ti{\ga}_e^2)^{\pr}|^2$ is small, and then apply the modified Wirtinger inequality to $\big(\ti{\ga}_e^1(t)-\ti{\ga}_e^2(t)-[\ti{\ga}_e^1(0)-\ti{\ga}_e^2(0)]\big)$. 

After the reparametrization, the constant speed curves $\ti{\ga}_e^i$ are geodesic segments on the intervals $I^i_j=[\frac{1}{S^i}\sum_{l<j}a^i_l, \frac{1}{S^i}\sum_{l\leq j}a^i_l]$, and $\ti{\ga}_e^i=\ga_e\circ P^i_j$ on each interval $I^i_j$ with $a^i_j\neq 0$, where the reparametrization $P^i_j: I^i_j\rightarrow [x_{2j}, x_{2j+2}]$ are just linear maps. Clearly $I^1_j$, $I^2_j$ and $P^1_j$, $P^2_j$ are close respectively. %are continuous with respect to $\ga^i$.
Let $I_j=I^1_j\cap I^2_j$, then $I\setminus\{\cup_{j=0}^{L-1}I_j\}$ is very small. Using the fact that $\ti{\ga}_e^i$ have constant speed with energy less than $L$, we have $|(\ti{\ga}_e^i)^{\pr}|\leq L$, so
$$\int_{I\setminus\cup I_j}|(\ti{\ga}_e^1-\ti{\ga}_e^2)^{\pr}|^2\leq 4L^2\cdot \Length(I\setminus\cup_j I_j),$$
which is hence small.

Given $\ep>0$ small enough, we can divide $I_j$ into two sub-classes. If $a^1_j<\ep$, we can assume that $a^2_j<2\ep$ by continuity, then
$$\int_{I_j}|(\ti{\ga}_e^1-\ti{\ga}_e^2)^{\pr}|^2\leq 2\int_{I^1_j}|(\ti{\ga}_e)^{\pr}|^2+2\int_{I^2_j}|(\ti{\ga}_e)^{\pr}|^2\leq 2L^2(a^1_j+a^2_j)<6L^2\ep.$$
If $a^1_j\geq\ep$, we can assume that $a^2_j\geq\frac{\ep}{2}$ by continuity, then $|(P^i_j)^{\pr}|=\frac{S^i}{a^i_j}$, and
$$\int_{I_j}|(\ti{\ga}_e^1-\ti{\ga}_e^2)^{\pr}|^2=\int_{I_j}\big|((\ga_e^1)^{\pr}\circ P^1_j)(P^1_j)^{\pr}-((\ga_e^2)^{\pr}\circ P^2_j)(P^2_j)^{\pr}\big|^2,$$
which can be made small as $\ga_e^i$ are close in $C^1$-norm and bounded in $C^2$-norm\footnote{This is because $|\ga_e^{\pr\pr}|\leq \sup|A_M|\cdot|\ga_e^{\pr}|^2\leq\frac{1}{16}L^2$ by  the geodesic equation and (M1) in \S \ref{introduction}.}, and $P^i_j$, $|(P^i_j)^{\pr}|$ are close as linear maps and numbers respectively.
\end{proof}

%%%%%%%%%%%%%%%%%%%%%%%%%%%%%%%%%%%%%%%%%%%%%%%%%%%%%%%%%%%

\subsection{Property (2) of $\Psi$.} We only need to prove that $\Psi$ preserves the homotopy class, as the length decreasing property is trivially true.
%The homotopy equivalence of $\ga$, $\ga_e$ and $\ti{\ga}_e$ is a corollary of Lemma \ref{homotopy equivalence}. The homotopy equivalence of $\ti{\ga}_e$, $\ga_o$ and $\ti{\ga}_o$ is easier, as there is no free boundary replacement, and the proof is contained in Lemma \ref{homotopy equivalence}.
In fact, this is just a corollary of Lemma \ref{homotopy equivalence} and Remark \ref{remark of homotopy equivalence} in the next section.

%%%%%%%%%%%%%%%%%%%%%%%%%%%%%%%%%%%%%%%%%%%%%%%%%%%%%%%%%%%%%%%%%%%%%%%%%%%%%%%%%%%%%%%%%%%%%%%%%%%%%%%%%%%%%%%%%%%%%%%

\section{Good sweepouts and free boundary min-max geodesics}\label{existence}

In this section, we will discuss apply the curve shortening process $\Psi$ to sweepouts in Definition \ref{variational space1} and prove Theorem \ref{main theorem2}. The foremost interesting question is when the width $W$, defined by (\ref{width}) corresponding to a sweepout $\si_0$,  is positive. In fact, $W$ is positive when $\si_0$ represents a nontrivial homotopy class in $\Om$. A special case is when the constraint submanifold $N$ bounds a non-contractable disk (among all disks with boundary lying on $N$). Suppose in contradiction that the width is small enough, then there exists a sweepout $\si$ homotopic to $\si_0$, and the energy of each $\si(t, \cdot)$ is very small, hence $\si(t, \cdot)$ all lie in a convex geodesic neighborhood of $N$. Property (M2)$^{\pr}$ implies that we can continuously shrink $\si(t, \cdot)$ to a point through curves with boundary lying on $N$. In fact, we can shrink $\si(\cdot, t)$ in a continuous way with respect to $``t"$ (e.g. using local coordinate charts), and homotopically deform $\si$ to a family of point curves, which contradicts the fact that $\si_0$ is homotopically nontrivial.

%%%%%%%%%%%%%%%%%%%%%%%%%%%%%%%%%%%%%%%%%%%%%%%%%%%%%%%%%%%

\subsection{Sweepouts and $\Psi$.}\label{sweepouts and Psi} Given a sweepout $\hat{\si}\in\Om$, we need to define a precise way to apply the curve shortening process $\Psi$ to deform $\hat{\si}$ to a sweep-out in $\La$. Assume that the maximal energy of slices $\hat{\si}(t)(\cdot)=\hat{\si}(t, \cdot)$ is bounded by a number $W_0$, i.e. $\max_{t\in[0, 1]}E(\hat{\si}(t))\leq W_0$, then the Cauchy-Schwartz inequality implies a uniform bound for the length and $C^{1/2}$-H{\"o}lder continuity for each $\hat{\si}(t)$, i.e.
$$d_M(\hat{\si}(t, x), \hat{\si}(t, y))\leq \Length(\hat{\si}(t)|_{[x, y]})=\int_x^y|\partial_x\hat{\si}(t, x)|dx$$
$$\leq |y-x|^{\frac{1}{2}}\big(\int_I |\partial_x\hat{\si}(t, x)|^2dx\big)^{\frac{1}{2}}\leq |x-y|^{\frac{1}{2}}W_0^{\frac{1}{2}}.$$
We will deform $\hat{\si}(t)$ to $\si(t)$ by Step 1 and Step 2 in \S \ref{defining Psi}. By the uniform $C^{1/2}$-H{\"o}lder bound, there exists an evenly spaced partition of $I=[0, 1]$ by $N$ points, i.e. $x_0=0, x_1, x_2, \cdots, x_N=1$, such that the length of $\hat{\si}(t)|_{[x_j, x_{j+1}]}$ is bounded by $2$ for all $t\in[0, 1]$, $j=0, 1, \cdots, N-1$. Hence we can apply Step 1 to $\hat{\si}(t)$ by (M2) and (M2)$^{\pr}$ to get $\si_e$, and then reparametrize $\si_e(t)$ to get a constant speed mapping $\ti{\si}_e(t)$, which we denote by $\si(t)$. By Lemma \ref{continuity of Psi1}, we know that $t\rightarrow \si(t)$ is a continuous mapping from $[0, 1]$ to $W^{1, 2}(I, M)$, and it is easy to see that $\si\in\Om$. Also the length bound of $\hat{\si}(t)$ implies a uniform Lipschitz bound of $\si(t)$ by $W_0^{1/2}$, as $\si(t)$ has constant speed and shorter than $\hat{\si}(t)$. Hence $\si(t)\in\La$ for some $L\in\N$, with $L$ larger than $N$ and $W_0^{1/2}$, so $\si(t)$ is a continuous path in $\La$. In the next lemma, we show that $\hat{\si}$ and $\si$ are homotopic.

\begin{lemma}\label{homotopy equivalence}
Given $\hat{\si}\in\Om$, and $\si_e$, $\ti{\si}_e$ as above, then $\hat{\si}$, $\si_e$ and $\ti{\si}_e$ are all homotopic in $\Om$.
\end{lemma}
\begin{proof}
First we show that $\hat{\si}$ is homotopic to $\si_e$ in $\Om$. Corresponding the $L$ break points\footnote{$L=N$ in the above case.}: $x_0=0, x_1, \cdots, x_L=1$, we deform $\hat{\si}$ to $\si_e$ in $L$ steps. For $s\in [0, x_1]$, define $F: [0, 1]\times I\times [0, x_1]\rightarrow M$, such that
\begin{displaymath}
F(t, x, s)=\left.\big\{ \begin{array}{ll}
H^{\pr}\big(\hat{\si}(t, s)\big)(x), \quad \textrm{if $0\leq x\le s$},\\
\hat{\si}(t, x), \quad\quad \textrm{if $s\leq x\leq 1$},
\end{array}\right.
\end{displaymath}
where $H^{\pr}$ is given by (\ref{definition of H prime})\footnote{Here the defining interval for $H^{\pr}(\cdot)$ is $[0, s]$, with $0\leq s\leq x_1$.}. (C2)$^{\pr}$ and Lemma \ref{boundary continuity in free boundary case} imply that $F$ is continuous, hence $F(\cdot, \cdot, s)$ lies in $\Om$ by the definition of $H^{\pr}$. So $F$ defines homotopy between $\hat{\si}(t, x)=F(t, x, 0)$ with $\si_1(t, x)=F(t, x, x_1)$ in $\Om$. Now we inductively define $F(t, x, s)$ as homotopy between $\si_j(t, x)$ and $\si_{j+1}(t, x)$ when $s\in [x_j, x_{j+1}]$, $1\leq j\leq L-2$. Assume that $\si_1, \cdots, \si_j$, and $F(\cdot, \cdot, s)$ are well-defined for $s\in[0, x_j]$, where $F(\cdot, \cdot, s): [x_l, x_{l+1}]\rightarrow\Om$ is a homotopy between $\si_l$ and $\si_{l+1}$, $1\leq l\leq j-1$. For $s\in[x_j, x_{j+1}]$, define $F: [0, 1]\times I\times [x_j, x_{j+1}]\rightarrow M$ by
\begin{displaymath}
F(t, x, s)=\left.\big\{ \begin{array}{ll}
H\big(\si_j(t, x_j), \si_j(t, s)\big)(x), \quad \textrm{if $x_j\leq x\le s$},\\
\si_j(t, x), \quad\quad \textrm{if $0\leq x\leq x_j$, or $s\leq x\leq 1$},
\end{array}\right.
\end{displaymath}
where $H$ is defined by (\ref{definition of H})\footnote{Here the defining interval for $H(\cdot, \cdot)$ is $[x_j, s]$, with $x_j\leq s\leq x_{j+1}$.}. (C2) implies that $F$ is continuous, so $F(\cdot, \cdot, s)\in\Om$, and $F$ defines a homotopy between $\si_j(t, x)=F(t, x, x_j)$ and $\si_{j+1}(t, x)=F(t, x, x_{j+1})$. Finally, using similar argument, we can use $H^{\pr}$ to define a homotopy between $\si_{L-1}$ with $\si_e$ by $F: [0, 1]\times I\times [x_{L-1}, 1]\rightarrow M$, with
\begin{displaymath}
F(t, x, s)=\left.\big\{ \begin{array}{ll}
H^{\pr}\big(\si_{L-1}(t, s)\big)(x), \quad \textrm{if $x_{L-1}+1-s \leq x\le 1$},\\
\si_{L-1}(t, x), \quad\quad \textrm{if $0\leq x\leq x_{L-1}+1-s$}.
\end{array}\right.
\end{displaymath}

Next we show that $\si_e$ is homotopic to $\ti{\si}_e$ in $\Om$. As $\ti{\si}_e(t, \cdot)$ is a reparametrization of $\si_e(t, \cdot)$, we can write $\si_e(t, \cdot)=\ti{\si}_e(t, \cdot)\circ P_t$, where $P_t: [0, 1]\rightarrow [0, 1]$ is a monotone piecewise linear map. By the proof of Lemma \ref{continuity of Psi1}, $P_t$ depends continuously on $``t"$\footnote{$P_t |_{[x_{2j}, x_{2j+2}]}=(P_j(t))^{-1}$ is a linear map from $[x_{2j}, x_{2j+2}]$ to $I_j(t)=[\frac{1}{S(t)}\sum_{l<j}a_l(t), \frac{1}{S(t)}\sum_{l\leq j}a_l(t)]$ using notations in the proof of Lemma \ref{continuity of Psi1}.}. Then
$$G(t, x, s)=\ti{\si}_e\big(t, (1-s)P_t(x)+sx\big),$$
is a homotopy between $G(\cdot, \cdot, 0)=\si_e$ and $G(\cdot, \cdot, 1)=\ti{\si}_e$. When $\Length(\ga_e(t, \cdot))=0$, we can let $P_t=id$, and $G$ is well-defined.
\end{proof}

\begin{remark}\label{remark of homotopy equivalence}
Applying Step 3 and Step 4 of $\Psi$ in \S \ref{defining Psi} to $\ti{\si}_e(t)$ gives $\si_o(t)$ and $\ti{\si}_o(t)$, with $\ti{\si}_o(t)\in\La$ for each $t$. $\si_o(t)$ an $\ti{\si}_o(t)$ are both continuous mapping from $[0, 1]$ to $W^{1, 2}(I, M)$ by Property (1) of $\Psi$, and it is easily seen that $\ti{\si}_o(t)\in\Om$. The homotopy equivalence of $\ti{\si}_e$, $\si_o$ and $\ti{\si}_o$ in $\Om$ follows from an easy modification of the above proof, as there is no boundary replacement here.
\end{remark}

%%%%%%%%%%%%%%%%%%%%%%%%%%%%%%%%%%%%%%%%%%%%%%%%%%%%%%%%%%%

\subsection{Almost maximal implies almost critical.}

We prove Theorem \ref{main theorem2} in this section.

\begin{proof}
(Theorem \ref{main theorem2}) Take a sequence $\{\hat{\si}_j\}\subset [\si_0]$, such that
$$\max_{t\in [0, 1]}E\big(\hat{\si}_j(t)\big)\leq W_0+\frac{1}{j}.$$
By the discussion in \S \ref{sweepouts and Psi}, we can deform $\hat{\si}_j$ to $\si_j$, such that $E(\si_j(t))\leq E(\hat{\si}_j(t))$, $\si_j(t)$ is a continuous path in $\La$, and $\si_j(t)$ is homotopic to $\hat{\si}_j(t)$ in $\Om$ by Lemma \ref{homotopy equivalence}, hence $\si_j(t)\in[\si_0]$. So $\max_{t\in [0, 1]}E\big(\si_j(t)\big)\leq \max_{t\in [0, 1]}E\big(\hat{\si}_j(t)\big)\leq W_0+\frac{1}{j}$.

Now apply the modified Birkhoff curve shortening map to each $\si_j(t)$ to get $\ga_j(t)=\Psi(\si_j(t))\in\La$, hence $\ga_j(t)$ is homotopic to $\si_j(t)$ in $\Om$ by Lemma \ref{homotopy equivalence} and Remark \ref{remark of homotopy equivalence}, so $\ga_j\in[\si_0]$. As $E(\ga_j(t))=\Length^2(\ga_j(t))\leq \Length^2(\si_j(t))=E(\si_j(t))$,
$$\max_{t\in [0, 1]}E\big(\ga_j(t)\big)\leq W_0+\frac{1}{j}.$$
So $lim_{j\rightarrow\infty}\max_{t\in [0, 1]}E(\ga_j(t))=W_0=W([\si_0])$.

We will show that $\{\ga_j\}$ satisfies the requirement of Theorem \ref{main theorem2}. If not, then there exist an $\epsilon>0$, a sequence $\de_i>0$, with $\lim_{i\rightarrow\infty}\de_i=0$, a subsequence $\ga_{j_i}$, with $j_i>\frac{1}{\de_i}$, and a sequence of $t_i\in[0, 1]$, with $E(\ga_{j_i}(t_i))=\Length^2(\ga_{j_i}(t_i))>W_0-\de_i$, but $dist(\ga_{j_i}(t_i), G)\geq \ep$. Also, as $E(\ga_j(t_i))\leq E(\si_j(t_i))\leq W_0+\frac{1}{j_i}\leq W_0+\de_i$, we have
$$W_0-\de_i\leq \Length^2\big(\ga_{j_i}(t_i)\big)\leq \Length^2\big(\si_{j_i}(t_i)\big)\leq W_0+\de_i.$$
Now denote $\ga_i=\ga_{j_i}(t_i)$ and $\si_i=\si_{j_i}(t_i)$. Since $W_0>0$, then $Length(\ga_i)\geq\frac{\ep}{2}$ for $i$ large enough. Then Property (3) of $\Psi$ implies that $dist(\ga_i, \si_i)\leq \frac{\ep}{2}$ for $i$ large enough, as $\ga_i=\Psi(\si_i)$, hence
$$dist(\si_i, G)\geq dist(\ga_i, G)-dist(\ga_i, \si_i)\geq \frac{\ep}{2}.$$
But this is a contradiction to Property (5) of $\Psi$, as $\Length(\si_i)-\Length(\ga_i)\rightarrow 0$ when $i\rightarrow\infty$.
\end{proof}

%%%%%%%%%%%%%%%%%%%%%%%%%%%%%%%%%%%%%%%%%%%%%%%%%%%%%%%%%%%%%%%%%%%%%%%%%%%%%%%%%%%%%%%%%%

%Add reference here
%\newpage

\parindent 0ex
%Mathematical Science Research Institute\\
%E-mail: xzhou@msri.org
Massachusetts Institute of Technology\\
Department of Mathematics\\
77 Massachusetts Avenue\\
Cambridge, MA 02139-4307\\
Email address: xinzhou@math.mit.edu

\end{document}